\documentclass[11pt]{article}
\usepackage[a4paper,margin=28mm]{geometry}
\usepackage{amsmath,amssymb,amsthm,mathtools}
\usepackage{fontspec,microtype,booktabs}
\usepackage[colorlinks=true,linkcolor=blue!55!black,citecolor=blue!55!black,
urlcolor=blue!55!black]{hyperref}
\usepackage[
  backend=biber,
  style=numeric-comp,
  maxbibnames=99,
  giveninits=true,
  isbn=false,
]{biblatex}

\DeclareFieldFormat*{title}{\mkbibemph{#1}}
\DeclareFieldFormat{journaltitle}{#1}
\DeclareFieldFormat{booktitle}{#1}
\renewbibmacro{in:}{\ifentrytype{article}{}{\printtext{\bibstring{in}\intitlepunct}}}
\usepackage{xcolor}
\numberwithin{equation}{section}
\newtheorem{theorem}{Theorem}[section]
\newtheorem{lemma}[theorem]{Lemma}
\newtheorem{proposition}[theorem]{Proposition}

\theoremstyle{remark}
\newtheorem{remark}[theorem]{Remark}
\newcommand{\KR}{\mathsf{KR}}
\newcommand{\ord}{\operatorname{ord}}
\title{The three Kanade--Russell identities modulo nine}
\author{Yuma Mizuno}
\date{October 5, 2026}
\begin{document}
\maketitle
\begin{abstract}
We prove the three symmetric Kanade--Russell identities modulo nine.
\end{abstract}

\section{Introduction}

\begin{theorem}\label{thm:KR-evaluations}
As formal power series in \(q\), and as holomorphic identities on
\(|q|<1\),
\begin{equation}\label{eq:KR-evaluations}
\begin{alignedat}{2}
 &\sum_{m,n\geq0}\frac{q^{m^2+3mn+3n^2\phantom{+0n+0n}}}{(q;q)_m(q^3;q^3)_n}
 &&=\frac1{(q,q^3,q^6,q^8;q^9)_\infty},\\
 &\sum_{m,n\geq0}\frac{q^{m^2+3mn+3n^2+m+3n}}{(q;q)_m(q^3;q^3)_n}
 &&=\frac1{(q^2,q^3,q^6,q^7;q^9)_\infty},\\
 &\sum_{m,n\geq0}\frac{q^{m^2+3mn+3n^2+2m+3n}}{(q;q)_m(q^3;q^3)_n}
 &&=\frac1{(q^3,q^4,q^5,q^6;q^9)_\infty},
\end{alignedat}
\end{equation}
where
\[
 (a;q)_n=\prod_{j=0}^{n-1}(1-aq^j),\qquad
 (a_1,\ldots,a_s;q)_\infty=\prod_{i=1}^s\prod_{j\geq0}(1-a_iq^j).
\]
\end{theorem}

Kanade and Russell used symbolic computation to discover new
Rogers--Ramanujan type partition identities and proposed conjectures
modulo nine and twelve in 2014 \cite{KanadeRussell2015}.
Subsequent work established several of their conjectures; in particular,
Bringmann, Jennings-Shaffer, and Mahlburg proved seven identities modulo
twelve \cite{BringmannJenningsShafferMahlburg2020}.
The three partition conjectures modulo nine corresponding to
\eqref{eq:KR-evaluations}, however,
remained unproved for more than a decade, despite the close resemblance
of their partition conditions and product forms to those of the classical
Rogers--Ramanujan identities.
For work on these conjectures, see, for example,
\cite{PennSadowski2017,ButoracSadowski2019,Kursungoz2019,PennSadowskiWebb2019,Kanade2019,
ChernLi2020,Chern2020,UncuZudilin2021,BakerKanadeRussellSadowski2023,
Konenkov2024,Tsuchioka2022,Mizuno2025,WangZhang2025,SunWang2026}.
In this paper, we prove \eqref{eq:KR-evaluations}.

Kanade and Russell originally formulated these conjectures as equalities
between numbers of partitions \cite{KanadeRussell2015}. For \(i=1,2,3\),
let \(\mathcal D_i\) consist of partitions
\(\lambda_1\leq\cdots\leq\lambda_\ell\) whose parts are at least
\(i\), such that
\(\lambda_{j+2}-\lambda_j\geq3\), and such that
\(\lambda_j+\lambda_{j+1}\equiv0\pmod3\) whenever
\(\lambda_{j+1}-\lambda_j\leq1\).
They conjectured that, for every \(N\geq0\), the
number of partitions of \(N\) in \(\mathcal D_i\) equals the number
of partitions of \(N\) into parts congruent to \(2^{i-1},3,6,9-2^{i-1}\pmod9\).
Kur\c{s}ung\"oz later expressed the generating functions for
\(\mathcal D_1,\mathcal D_2,\mathcal D_3\) as the three double sums in
\eqref{eq:KR-evaluations}, respectively
\cite{Kursungoz2019}.
Thus these partition conjectures are equivalent to
\eqref{eq:KR-evaluations}.

Wang and Zhang \cite[Theorem 1.1]{WangZhang2025} proved that the triple of
product sides in \eqref{eq:KR-evaluations}, suitably normalized,
is a weight-zero vector-valued modular function on \(\Gamma_0(3)\).
By Theorem~\ref{thm:KR-evaluations}, the same holds for the triple of sum sides.

We now outline the proof. Define
\begin{equation}\label{eq:source-definition}
 \mathsf F(x,y)=\sum_{m,n\geq0}
 \frac{q^{m^2+3mn+3n^2}x^my^n}{(q;q)_m(q^3;q^3)_n},\qquad
 (\mathsf A,\mathsf B,\mathsf C)
 =\bigl(\mathsf F(1,1),\mathsf F(q,q^3),\mathsf F(q^2,q^3)\bigr),
\end{equation}
so that \(\mathsf A,\mathsf B,\mathsf C\) are the three sums in \eqref{eq:KR-evaluations},
and write \(\KR_1,\KR_2,\KR_3\) for the corresponding products.

The key is that one can show, without using \eqref{eq:KR-evaluations}, that the triple of sums \((\mathsf A,\mathsf B,\mathsf C)\)
and the triple of products \((\KR_1,\KR_2,\KR_3)\) give the same value
when substituted into the cubic polynomial
\[
 \mathcal N(X,Y,Z)=X^3+qY^3-q^2Z^3+3qXYZ.
\]
By a result of Tsuchioka~\cite{Tsuchioka2022}, each coefficient of
\(\mathsf A\), \(\mathsf B\), \(\mathsf C\) is at least the corresponding coefficient of
\(\KR_1\), \(\KR_2\), \(\KR_3\), respectively.
These inequalities and the equality of the values of \(\mathcal N\)
imply \eqref{eq:KR-evaluations} (Section~\ref{sec:positive-completion}).

The two values are expressed through a double-sum theta series as follows:
\begin{equation}\label{eq:equal-norms}
 \mathcal N(\mathsf A,\mathsf B,\mathsf C)
 =\frac{\mathfrak a(q)}{(q;q)_\infty(q^3;q^3)_\infty}
 =\mathcal N(\KR_1,\KR_2,\KR_3),
 \qquad
 \mathfrak a(q)=\sum_{r,s\in\mathbb Z}q^{r^2+rs+s^2}.
\end{equation}
The first equality in \eqref{eq:equal-norms} is proved as follows.
Let \(\mathsf G\) be the series obtained from \(\mathsf F\) by replacing \(3mn\) with \(-3mn\) in the exponent.
Section~\ref{app:addition} expresses three values of \(\mathsf G\) as quadratic polynomials in \(\mathsf A,\mathsf B,\mathsf C\).
Section~\ref{app:casoratian} writes \(\mathsf C\) as a \(2\times2\) determinant.
Section~\ref{app:Morita} relates this matrix to the values of \(\mathsf G\) by a connection formula.
Comparing determinants gives the first equality in \eqref{eq:equal-norms} (Lemma~\ref{lem:equal-norms}).
Section~\ref{sec:product} proves the second equality using
the Weierstrass addition formula and classical identities for theta and Lambert series.

The day after this manuscript was prepared\footnote{The chronology can be checked
against the timestamp of the immutable release at
\url{https://github.com/yuma-mizuno/kanade-russell-mod9}.}, Xia's preprint \cite{Xia2026} appeared on arXiv.
It independently proves all five Kanade--Russell identities modulo nine, namely \eqref{eq:KR-evaluations} and two asymmetric ones.
Xia's proof differs from ours: it considers sequences of series on the sum side and
on the product side, the latter constructed using divided differences, and identifies them
through a common recurrence and boundary conditions.

\paragraph*{Acknowledgments.}
The author thanks Ernest X. W. Xia for explaining how he arrived at his proof in \cite{Xia2026}.
The author was supported by the Irish Research Council Advanced Laureate Award
IRCLA/2023/1934 held by Robert Osburn.
The author used GPT-5.6 Sol and GPT-6 Astra in searching for the proof and preparing the manuscript.
The process is described in the final section, \nameref{sec:ai-discovery}.

\section{Bilinear identities}
\label{app:addition}

Define
\[
 \mathsf G(x,y)=\sum_{m,n\geq0}
 \frac{q^{m^2-3mn+3n^2}x^my^n}{(q;q)_m(q^3;q^3)_n}.
\]

\begin{proposition}[Bilinear identities]\label{prop:app-addition}
The following identities hold.
\begin{align}
 \mathsf G(x,1)
 & =\mathsf F(x,1)\mathsf F(x,x^3)
       +qx^2\mathsf F(qx,q^3)\mathsf F(q^2x,q^3x^3),
       \label{eq:app-addition-0}\\
 \mathsf G(qx,1)
 & =\mathsf F(x,1)\mathsf F(qx,q^3x^3)
       -qx\mathsf F(q^2x,q^3)\mathsf F(q^2x,q^3x^3),
       \label{eq:app-addition-plus}\\
 \mathsf G(q^{-1}x,q^3)
 & =x\mathsf F(qx,q^3)\mathsf F(qx,q^3x^3)
       +\mathsf F(q^2x,q^3)\mathsf F(x,x^3).
       \label{eq:app-addition-minus}
\end{align}
\end{proposition}

\begin{proof}
We first prove \eqref{eq:app-addition-0} by showing that both sides satisfy the same
\(q\)-difference equation and have the same constant and linear coefficients.
We then derive \eqref{eq:app-addition-plus} and \eqref{eq:app-addition-minus} using \(q\)-difference equations for \(\mathsf F\) and \(\mathsf G\).

For each integer \(j\), define
\[
 R_j=\sum_{n\geq0}\frac{q^{3n^2+3jn}}{(q^3;q^3)_n}.
\]
Termwise subtraction and a shift of \(n\) give
\begin{equation}\label{eq:app-R-contiguity}
 R_j-R_{j+1}=q^{3j+3}R_{j+2}.
\end{equation}
If \(g(x)=\mathsf G(x,1)=\sum_{j\geq0}g_jx^j\), then
\(g_j=q^{j^2}R_{-j}/(q;q)_j\).  Equation
\eqref{eq:app-R-contiguity} implies, for \(j\geq2\),
\[
 (1-q^j)(1-q^{j-1})g_j
 =q^{2j-1}(1-q^{j-1})g_{j-1}+q^{j-1}g_{j-2}.
\]
Equivalently,
\begin{equation}\label{eq:app-scalar-equation}
 qg(x)-(1+q+q^2x^2)g(qx)
       +(1-q^2x)g(q^2x)+q^2xg(q^3x)=0.
\end{equation}
For \(j\geq2\), the factor multiplying \(g_j\) in
\eqref{eq:app-scalar-equation} is \(q(1-q^j)(1-q^{j-1})\ne0\).
Thus a solution is uniquely determined by its constant and linear coefficients. For \(g\), we have
\begin{equation}\label{eq:app-g-initial}
 g(x)=R_0+\frac{q(R_0+R_1)}{1-q}x+O(x^2).
\end{equation}

Let \(h(x)\) be the right side of \eqref{eq:app-addition-0}.
We show that \(h\) satisfies \eqref{eq:app-scalar-equation} by writing
\(h(x)\), \(h(qx)\), and \(qxh(q^2x)\) as bilinear expressions in two vectors of values of \(\mathsf F\).
The series \(\mathsf F\) satisfies the \(q\)-difference equations
\begin{equation}\label{eq:app-F-contiguities}
\begin{aligned}
 \mathsf F(x,y)-\mathsf F(qx,y)
       &=qx\mathsf F(q^2x,q^3y),\\
 \mathsf F(x,y)-\mathsf F(x,q^3y)
       &=q^3y\mathsf F(q^3x,q^6y).
\end{aligned}
\end{equation}
For the column vector
\[
 \mathbf f(x)=\begin{pmatrix}
 \mathsf F(x,1)\\
 \mathsf F(qx,q^3)\\
 \mathsf F(q^2x,q^3)
 \end{pmatrix},
\]
they give
\begin{equation}\label{eq:app-L-system}
 \mathbf f(qx)=M(x)\mathbf f(x),\qquad
 M(x)=\begin{pmatrix}1&0&-qx\\0&0&1\\-x&x&1+qx^2\end{pmatrix}.
\end{equation}
For \(v(x)=\mathsf F(qx,x^3)\), they give
\begin{equation}\label{eq:app-p-recurrence}
 v(x)=(1-qx)v(qx)+qx(1+q+q^2x^2)v(q^2x)+q^4x^2v(q^3x)
\end{equation}
and
\begin{equation}\label{eq:app-p-bridges}
\begin{aligned}
 \mathsf F(x,x^3)&=v(x)+qx v(qx),\\
 \mathsf F(qx,q^3x^3)&=v(qx)+q^2x v(q^2x),\\
 \mathsf F(q^2x,q^3x^3)&=v(qx).
\end{aligned}
\end{equation}
By \eqref{eq:app-p-recurrence}, the column vector
\(\mathbf v(x)=(v(x),v(qx),v(q^2x))^{\mathsf T}\) satisfies
\begin{equation}\label{eq:app-T-system}
 \mathbf v(x)=T(x)\mathbf v(qx),\qquad
 T(x)=\begin{pmatrix}
 1-qx&qx(1+q+q^2x^2)&q^4x^2\\1&0&0\\0&1&0
 \end{pmatrix}.
\end{equation}

Define three polynomial matrices by
\[
 K_0=\begin{pmatrix}1&qx&0\\0&qx^2&0\\0&0&0\end{pmatrix},\qquad
 K_1=\begin{pmatrix}0&1&q^2x\\0&0&0\\0&-qx&0\end{pmatrix},\qquad
 K_2=\begin{pmatrix}1&qx-1&-q^2x\\
                         0&0&-q^3x^3\\-qx&qx-q^2x^2&0\end{pmatrix}.
\]
Then
\begin{equation}\label{eq:app-h-bilinear}
\begin{aligned}
 h(x)&=\mathbf f(x)^{\mathsf T}K_0(x)\mathbf v(x),\\
 h(qx)&=\mathbf f(x)^{\mathsf T}K_1(x)\mathbf v(x),\\
 qx h(q^2x)&=\mathbf f(x)^{\mathsf T}K_2(x)\mathbf v(x).
\end{aligned}
\end{equation}
The first identity in \eqref{eq:app-h-bilinear} follows from
\eqref{eq:app-p-bridges}. For the second and third, direct computation gives
\begin{equation}\label{eq:app-bilinear-certificate}
\begin{aligned}
 K_1(x)T(x)&=M(x)^{\mathsf T}K_0(qx),\\
 K_2(x)T(x)&=qxM(x)^{\mathsf T}K_1(qx).
\end{aligned}
\end{equation}
Multiplying these on the left by
\(\mathbf f(x)^{\mathsf T}\) and on the right by \(\mathbf v(qx)\), and using
\eqref{eq:app-L-system} and \eqref{eq:app-T-system}, gives the second and third
identities in \eqref{eq:app-h-bilinear}.

Direct computation also gives
\begin{equation}\label{eq:app-scalar-certificate}
\begin{aligned}
 \bigl(q^2xK_0(x)-qx(1+q+q^2x^2)K_1(x)
       &+(1-q^2x)K_2(x)\bigr)T(x)\\
       &+qxM(x)^{\mathsf T}K_2(qx)=0.
\end{aligned}
\end{equation}
Multiplying this on the left by
\(\mathbf f(x)^{\mathsf T}\) and on the right by \(\mathbf v(qx)\), and using
\eqref{eq:app-L-system}, \eqref{eq:app-T-system}, and \eqref{eq:app-h-bilinear},
gives \(qx\) times the left side of \eqref{eq:app-scalar-equation} with \(h\) in place of \(g\).
Hence \(h\) satisfies \eqref{eq:app-scalar-equation}.

The second summand on the right side of \eqref{eq:app-addition-0} starts at \(x^2\).
Expanding the two factors of the first summand from their defining sums gives
\[
\begin{aligned}
 h(x)
 &=\left(R_0+\frac{qR_1}{1-q}x+O(x^2)\right)
   \left(1+\frac{q}{1-q}x+O(x^2)\right)+O(x^2)\\
 &=R_0+\frac{q(R_0+R_1)}{1-q}x+O(x^2).
\end{aligned}
\]
Comparing with \eqref{eq:app-g-initial}, \(h\) and \(g\) have the same
constant and linear coefficients. Since \(h\) and \(g\) satisfy
\eqref{eq:app-scalar-equation}, uniqueness gives \(h=g\).

The second identity in \eqref{eq:app-h-bilinear} and \eqref{eq:app-p-bridges} now give
\eqref{eq:app-addition-plus}. For the last identity, termwise
subtraction gives \(g(x)-g(qx)=qx\mathsf G(q^2x,q^{-3})\).
The relation \(R_{-j-1}=R_{-j}+q^{-3j}R_{1-j}\), from
\eqref{eq:app-R-contiguity}, then gives
\[
 qx\mathsf G(q^{-1}x,q^3)=g(x)-g(qx)-qxg(q^2x).
\]
Finally, multiplying the matrix identity
\[
 K_0-K_1-K_2
       =qx\begin{pmatrix}0&0&0\\0&x&q^2x^2\\1&qx&0\end{pmatrix}
\]
on the left by \(\mathbf f(x)^{\mathsf T}\) and on the right by \(\mathbf v(x)\),
using \eqref{eq:app-h-bilinear} and \eqref{eq:app-p-bridges}, and cancelling \(qx\) proves
\eqref{eq:app-addition-minus}.
\end{proof}

At \(x=1\), Proposition~\ref{prop:app-addition} yields
\begin{equation}\label{eq:app-quadratic-bridge}
 \mathsf G(1,1)=\mathsf A^2+q\mathsf B\mathsf C,\qquad
 \mathsf G(q^{-1},q^3)=\mathsf B^2+\mathsf A\mathsf C,\qquad
 \mathsf G(q,1)=\mathsf A\mathsf B-q\mathsf C^2.
\end{equation}
\begin{remark}
Expressions for the three sums on the left in terms of infinite products were conjectured
by Wang and Wang~\cite[Conjecture 3.7]{WangWang2026}.
The identities \eqref{eq:app-quadratic-bridge}, together with Theorem~\ref{thm:KR-evaluations}, prove their conjecture.
To match Wang and Wang's third formula, we rewrite its right-hand side
\(\KR_1^2\KR_3/\KR_2-2q\KR_3^2\) as \(\KR_1\KR_2-q\KR_3^2\)
using the classical theta identity
\[
 \KR_1\KR_2^2-\KR_1^2\KR_3+q\KR_2\KR_3^2=0.
\]
\end{remark}

\section{Casoratian}
\label{app:casoratian}

Define the \(q\)-Airy function by
\[
 \operatorname{Ai}_q(X)=\sum_{n\geq0}
       \frac{q^{n(n-1)/2}X^n}{(q^2;q^2)_n}.
\]
It satisfies
\begin{equation}\label{eq:app-Ai-contiguity}
 \operatorname{Ai}_q(X)-\operatorname{Ai}_q(q^2X)
       =X\operatorname{Ai}_q(qX).
\end{equation}
Consequently
\(\operatorname{Ai}_q(X)\operatorname{Ai}_q(-qX)+\operatorname{Ai}_q(-X)\operatorname{Ai}_q(qX)\)
is invariant under \(X\mapsto qX\). Comparing coefficients shows that it is constant,
and its value at \(X=0\) gives
\begin{equation}\label{eq:app-Ai-Wronskian}
 \operatorname{Ai}_q(X)\operatorname{Ai}_q(-qX)
       +\operatorname{Ai}_q(-X)\operatorname{Ai}_q(qX)=2.
\end{equation}

Define the series
\begin{align*}
 \mathcal A(x)&=\sum_{\mathclap{\substack{m,n\geq0\\ m\equiv n\ (\mathrm{mod}\ 2)}}}\;
 \frac{(-1)^m q^{(m^2+6mn+3n^2+6n)/4}x^m}{(q;q)_m(q^3;q^3)_n},\\
 \mathcal B(x)&=\sum_{\mathclap{\substack{m,n\geq0\\ m\not\equiv n\ (\mathrm{mod}\ 2)}}}\;
 \frac{(-1)^m q^{(m^2+6mn+3n^2+6n-1)/4}x^m}{(q;q)_m(q^3;q^3)_n}.
\end{align*}
Put
\begin{equation}\label{eq:parity-recombination}
 \mathcal U_\pm(x)=\mathcal A(x)\pm q^{1/4}\mathcal B(x).
\end{equation}

\begin{proposition}\label{prop:casoratian}
The series \(\mathcal A\) and \(\mathcal B\) satisfy
\begin{equation}\label{eq:app-Casoratian}
 \det\begin{pmatrix}
 \mathcal A(x)&\mathcal B(x)\\
 \mathcal A(qx)&\mathcal B(qx)
 \end{pmatrix}
 =x\mathsf F(q^2x,q^3x^3).
\end{equation}
In particular,
\begin{equation}\label{eq:source-minors}
 \det\begin{pmatrix}\mathcal A(1)&\mathcal B(1)\\\mathcal A(q)&\mathcal B(q)\end{pmatrix}=\mathsf F(q^2,q^3)=\mathsf C.
\end{equation}
\end{proposition}

\begin{proof}
Set \(c_m^\pm=\operatorname{Ai}_{q^{3/2}}(\pm q^{(6m+9)/4})\).
The coefficient of \(x^m\) in \(\mathcal U_\pm(x)\) is
\((\mp1)^m q^{m^2/4}c_m^\pm/(q;q)_m\).
Equation \eqref{eq:app-Ai-contiguity} gives
\(c_{m-2}^\pm-c_m^\pm=\pm q^{(6m-3)/4}c_{m-1}^\pm\), so these coefficients satisfy
the same recurrence as \(g_m\) in Section~\ref{app:addition}.
Hence \(\mathcal U_\pm\) satisfy \eqref{eq:app-scalar-equation},
and so do \(\mathcal A\) and \(\mathcal B\) by \eqref{eq:parity-recombination}.
Write \(\mathbf r(x)=(\mathcal A(x),\mathcal B(x))\).
Then \eqref{eq:app-Ai-Wronskian} gives
\begin{equation}\label{eq:app-parity-initial-determinant}
 \det\begin{pmatrix}\mathbf r(0)\\\mathbf r'(0)\end{pmatrix}
 =-\frac{
 \operatorname{Ai}_{q^{3/2}}(q^{9/4})\operatorname{Ai}_{q^{3/2}}(-q^{15/4})
 +\operatorname{Ai}_{q^{3/2}}(-q^{9/4})\operatorname{Ai}_{q^{3/2}}(q^{15/4})}
 {2(1-q)}=-\frac1{1-q}.
\end{equation}

For \(j=1,2\), define
\[
 W_j(x)=\det\begin{pmatrix}\mathbf r(x)\\\mathbf r(q^jx)\end{pmatrix}.
\]
Equation \eqref{eq:app-parity-initial-determinant} gives
\(W_1(x)=x+O(x^2)\). Put \(H(x)=W_1(x)/x\).
Since both components of \(\mathbf r\) satisfy \eqref{eq:app-scalar-equation},
\[
 q\mathbf r(x)-(1+q+q^2x^2)\mathbf r(qx)
 +(1-q^2x)\mathbf r(q^2x)+q^2x\mathbf r(q^3x)=0.
\]
Taking the left-hand side as the first row and \(\mathbf r(qx)\), \(\mathbf r(q^2x)\)
as the second row, respectively, and expanding the determinants gives
\begin{align*}
 q^2x W_2(qx)&=qW_1(x)-(1-q^2x)W_1(qx),\\
 qW_2(x)&=(1+q+q^2x^2)W_1(qx)+q^2xW_1(q^2x).
\end{align*}
Eliminating \(W_2\), substituting \(W_1(x)=xH(x)\),
and cancelling \(qx\) gives
\begin{equation}\label{eq:app-exterior-equation}
 H(x)=(1-q^2x)H(qx)
 +q^2x(1+q+q^4x^2)H(q^2x)+q^6x^2H(q^3x).
\end{equation}
For each positive degree \(j\), the new coefficient is multiplied
by \(1-q^j\ne0\).  Thus the coefficients of \(H\) are determined successively by \(H(0)=1\).
Equation \eqref{eq:app-p-recurrence}, with \(x\) replaced by \(qx\),
shows that \(\mathsf F(q^2x,q^3x^3)\) is a solution with constant term \(1\).
Hence \(H(x)=\mathsf F(q^2x,q^3x^3)\), which is \eqref{eq:app-Casoratian}.
Setting \(x=1\) gives \eqref{eq:source-minors}.
\end{proof}

\section{Connection formula}
\label{app:Morita}

Define the theta function by
\[
 \theta_q(z)=\sum_{k\in\mathbb Z}q^{k(k-1)/2}z^k.
\]
It satisfies
\begin{equation}\label{eq:formal-jacobi-shifts}
 \theta_q(qz)=z^{-1}\theta_q(z),\qquad
 \theta_q(z^{-1})=z^{-1}\theta_q(z),\qquad
 \theta_q(q/z)=\theta_q(z).
\end{equation}
Define the Ramanujan function by
\[
 \mathcal R_q(z)=\sum_{n\geq0}\frac{q^{n^2}(-z)^n}{(q;q)_n}.
\]
Written with the \(q\)-Airy function of Section~\ref{app:casoratian},
the connection formula \cite[Section~1, Theorem]{Morita2011} is
\begin{equation}\label{eq:app-Morita}
 \theta_q(X/q)\operatorname{Ai}_q(-X)
 +\theta_q(-X/q)\operatorname{Ai}_q(X)
 =2(q^2;q^2)_\infty\mathcal R_{q^2}(-q^3/X^2).
\end{equation}

\begin{proposition}\label{prop:connection}
The following identity holds.
\begin{equation}\label{eq:app-source-matrix}
 \begin{pmatrix}\mathcal A(1)&\mathcal B(1)\\\mathcal A(q)&\mathcal B(q)\end{pmatrix}
 \begin{pmatrix}\theta_{q^2}(1)&\theta_{q^6}(q^3)\\
                 \theta_{q^2}(q)&-q\theta_{q^6}(1)\end{pmatrix}
 =\begin{pmatrix}
 (q;q)_\infty\mathsf G(q,1)&(q^3;q^3)_\infty\mathsf G(1,1)\\
 (q;q)_\infty\mathsf G(q^{-1},q^3)&(q^3;q^3)_\infty\mathsf G(q,1)
 \end{pmatrix}.
\end{equation}
\end{proposition}

\begin{proof}
For \(j=0,1\), put \(X=q^{(6n+4j+1)/4}\) in \eqref{eq:app-Morita} for \(\theta_{q^{1/2}}\) and \(\operatorname{Ai}_{q^{1/2}}\).
Then multiply both sides by \(q^{3n^2+3jn}/(q^3;q^3)_n\) and sum over \(n\geq0\).
The exponent on the right becomes \(m^2-3mn+3n^2+{(1-2j)m}+3jn\),
which gives \(\mathsf G(q,1)\) and \(\mathsf G(q^{-1},q^3)\) for \(j=0,1\), respectively.
Applying to the left side
\[
 \theta_{q^{1/2}}(\pm q^{(6n+4j-1)/4})
       =(\pm1)^n q^{(-9n^2+(6-12j)n)/4}\theta_{q^{1/2}}(\pm q^{(4j-1)/4})
\]
and \(\theta_{q^{1/2}}(\pm q^{3/4})=\theta_{q^{1/2}}(\pm q^{-1/4})\), both from \eqref{eq:formal-jacobi-shifts}, we obtain
\begin{equation}\label{eq:app-short-Morita}
\begin{aligned}
 \theta_{q^{1/2}}(q^{-1/4})\mathcal U_+(1)
 +\theta_{q^{1/2}}(-q^{-1/4})\mathcal U_-(1)
       &=2(q;q)_\infty\mathsf G(q,1),\\
 \theta_{q^{1/2}}(q^{-1/4})\mathcal U_+(q)
 +\theta_{q^{1/2}}(-q^{-1/4})\mathcal U_-(q)
       &=2(q;q)_\infty\mathsf G(q^{-1},q^3).
\end{aligned}
\end{equation}

Next, put \(X=q^{(6m+9)/4}\) in \eqref{eq:app-Morita} for \(\theta_{q^{3/2}}\) and \(\operatorname{Ai}_{q^{3/2}}\).
Then multiply both sides by \(q^{m^2+jm}/(q;q)_m\) and sum over \(m\geq0\).
The exponent on the right becomes \(m^2-3mn+3n^2+jm\),
which gives \(\mathsf G(1,1)\) and \(\mathsf G(q,1)\) for \(j=0,1\), respectively.
Applying to the left side
\[
 \theta_{q^{3/2}}(\pm q^{(6m+3)/4})
       =(\pm1)^m q^{-3m^2/4}\theta_{q^{3/2}}(\pm q^{3/4})
\]
from \eqref{eq:formal-jacobi-shifts}, we obtain
\begin{equation}\label{eq:app-long-Morita}
\begin{aligned}
 \theta_{q^{3/2}}(-q^{3/4})\mathcal U_+(1)
 +\theta_{q^{3/2}}(q^{3/4})\mathcal U_-(1)
       &=2(q^3;q^3)_\infty\mathsf G(1,1),\\
 \theta_{q^{3/2}}(-q^{3/4})\mathcal U_+(q)
 +\theta_{q^{3/2}}(q^{3/4})\mathcal U_-(q)
       &=2(q^3;q^3)_\infty\mathsf G(q,1).
\end{aligned}
\end{equation}

Even and odd index separation gives
\[
 \theta_{q^{1/2}}(\pm q^{-1/4})
       =\theta_{q^2}(1)\pm q^{-1/4}\theta_{q^2}(q),\qquad
 \theta_{q^{3/2}}(\pm q^{3/4})
       =\theta_{q^6}(q^3)\pm q^{3/4}\theta_{q^6}(1).
\]
Substituting these and \eqref{eq:parity-recombination} into
\eqref{eq:app-short-Morita}--\eqref{eq:app-long-Morita}
proves the proposition.
\end{proof}

\section{A cubic identity for the products}\label{sec:product}

\begin{proposition}\label{prop:mc-product-cubic}
The three products \(\KR_1,\KR_2,\KR_3\) satisfy
\begin{equation}\label{eq:mc-product-norm}
 \mathcal N(\KR_1,\KR_2,\KR_3)
 =\frac{\mathfrak a(q)}{(q;q)_\infty(q^3;q^3)_\infty}.
\end{equation}
\end{proposition}

Define the Lambert series by
\[
 L(q)=\sum_{n\ge1}\frac{q^n}{(1-q^n)^2}.
\]
We use the square identity \cite[Theorem 7.2]{Matsuda2020}
\begin{equation}\label{eq:formal-square}
 \mathfrak a(q)^2=1+12L(q)-36L(q^3)
\end{equation}
and the trisection identity \cite[(2.4), with $q$ replaced by $q^3$]{BorweinBorweinGarvan1994}
\begin{equation}\label{eq:formal-trisection}
 (\mathfrak a(q)-\mathfrak a(q^3))(q^3;q^3)_\infty
 =6q(q^9;q^9)_\infty^3.
\end{equation}
The Jacobi triple product identity gives
\begin{equation}\label{eq:formal-jacobi}
 \theta_q(z)=(q,-z,-q/z;q)_\infty,
\end{equation}
and the Weierstrass addition formula gives
\begin{equation}\label{eq:formal-addition}
 \begin{aligned}
 &\theta_q(-xz)\theta_q(-x/z)\theta_q(-y)^2
 -\theta_q(-yz)\theta_q(-y/z)\theta_q(-x)^2\\
 &\hspace{15mm}=\frac yz \theta_q(-xy)\theta_q(-x/y)\theta_q(-z)^2.
 \end{aligned}
\end{equation}
Define
\[
 \Sigma(u;q)=\sum_{n\ge0}\left(
 \frac{uq^n}{(1-uq^n)^2}
 +\frac{q^{n+1}/u}{(1-q^{n+1}/u)^2}\right).
\]
Expanding the paired factors of the product \eqref{eq:formal-jacobi} at $z=1+\varepsilon$ gives
\begin{equation}\label{eq:formal-kernel}
 \theta_q(-u(1+\varepsilon))\theta_q(-u/(1+\varepsilon))
 =\theta_q(-u)^2\bigl(1-\Sigma(u;q)\varepsilon^2\bigr)
 \pmod{\varepsilon^3},
\end{equation}
since $(1-bz)(1-b/z)/(1-b)^2=1-b(z+z^{-1}-2)/(1-b)^2$ and $z+z^{-1}-2=\varepsilon^2\pmod{\varepsilon^3}$.

\begin{proof}[Proof of Proposition~\ref{prop:mc-product-cubic}]
For $1\le r<9$, write $J_r=\theta_{q^9}(-q^r)$ and $\Lambda_r=\Sigma(q^r;q^9)$. Splitting the products by residue classes gives
\begin{gather*}
 J_3=(q^3;q^3)_\infty,\qquad
 J_1J_2J_4=\frac{(q;q)_\infty(q^9;q^9)_\infty^3}{(q^3;q^3)_\infty},\\
 (\KR_1,\KR_2,\KR_3)
 =\frac{(q^9;q^9)_\infty^2}{(q^3;q^3)_\infty}
 \left(\frac1{J_1},\frac1{J_2},\frac1{J_4}\right),\\
 \KR_1\KR_2\KR_3=\frac{(q^9;q^9)_\infty^3}{(q;q)_\infty(q^3;q^3)_\infty^2}.
\end{gather*}
In \eqref{eq:formal-addition} for $\theta_{q^9}$, put $x=q^r$, $y=q^3$, and $z=1+\varepsilon$. By \eqref{eq:formal-kernel}, the coefficient of $\varepsilon^2$ gives
\[
 J_r^2(q^3;q^3)_\infty^2(\Lambda_r-\Lambda_3)
 =-q^3(q^9;q^9)_\infty^6\theta_{q^9}(-q^{r+3})\theta_{q^9}(-q^{r-3}).
\]
For $r=1,2,4$, \eqref{eq:formal-jacobi-shifts} and the preceding product relations simplify these identities to
\begin{align*}
 \Lambda_1-\Lambda_3&=q(q;q)_\infty(q^9;q^9)_\infty^3\KR_1^3,\\
 \Lambda_2-\Lambda_3&=q^2(q;q)_\infty(q^9;q^9)_\infty^3\KR_2^3,\\
 \Lambda_4-\Lambda_3&=-q^3(q;q)_\infty(q^9;q^9)_\infty^3\KR_3^3.
\end{align*}
Their sum has left side $L(q)-4L(q^3)+3L(q^9)$. Multiplying by $(q^3;q^3)_\infty^2$ and using the product of the three $\KR_i$ gives
\begin{equation}\label{eq:formal-norm-reduction}
 \begin{aligned}
 &(q^3;q^3)_\infty^2\bigl(L(q)-4L(q^3)+3L(q^9)\bigr)\\
 &\quad=q(q;q)_\infty(q^3;q^3)_\infty^2(q^9;q^9)_\infty^3
 \mathcal N(\KR_1,\KR_2,\KR_3)
 -3q^2(q^9;q^9)_\infty^6.
 \end{aligned}
\end{equation}

It remains to evaluate the Lambert expression. Subtracting \eqref{eq:formal-square} at $q^3$ from the same identity at $q$ yields
\[
 12\bigl(L(q)-4L(q^3)+3L(q^9)\bigr)
 =(\mathfrak a(q)-\mathfrak a(q^3))(\mathfrak a(q)+\mathfrak a(q^3)).
\]
Applying \eqref{eq:formal-trisection} and writing the sum of the two theta series as $2\mathfrak a(q)-(\mathfrak a(q)-\mathfrak a(q^3))$ gives
\[
 (q^3;q^3)_\infty^2\bigl(L(q)-4L(q^3)+3L(q^9)\bigr)
 =q\mathfrak a(q)(q^3;q^3)_\infty(q^9;q^9)_\infty^3
 -3q^2(q^9;q^9)_\infty^6.
\]
Comparison with \eqref{eq:formal-norm-reduction} and cancellation proves the proposition.
\end{proof}

\section{Proof of the main theorem}
\label{sec:positive-completion}

\begin{lemma}\label{lem:equal-norms}
The three sums \(\mathsf A,\mathsf B,\mathsf C\) satisfy
\[
 \mathcal N(\mathsf A,\mathsf B,\mathsf C)
 =\frac{\mathfrak a(q)}{(q;q)_\infty(q^3;q^3)_\infty}.
\]
\end{lemma}

\begin{proof}
Splitting the summation lattice according to the parity of \(s\)
gives
\begin{equation}\label{eq:a2-parity}
 \mathfrak a(q)=\theta_{q^2}(q)\theta_{q^6}(q^3)
                   +q\theta_{q^2}(1)\theta_{q^6}(1).
\end{equation}
Indeed, the substitutions \((r,s)=(\ell-k,2k)\) and
\((r,s)=(\ell-k,2k-1)\) give the exponents
\(\ell^2+3k^2\) and
\(\ell(\ell-1)+3k(k-1)+1\), respectively. Taking determinants on both sides of
\eqref{eq:app-source-matrix} in Proposition~\ref{prop:connection}, and using \eqref{eq:source-minors}
and \eqref{eq:a2-parity}, gives
\[
 \mathsf G(q,1)^2-\mathsf G(1,1)\mathsf G(q^{-1},q^3)
 =-\frac{\mathsf C\mathfrak a(q)}{(q;q)_\infty(q^3;q^3)_\infty}.
\]
Substituting \eqref{eq:app-quadratic-bridge} yields
\[
 (\mathsf A\mathsf B-q\mathsf C^2)^2
 -(\mathsf A^2+q\mathsf B\mathsf C)(\mathsf B^2+\mathsf A\mathsf C)
 =-\frac{\mathsf C\mathfrak a(q)}{(q;q)_\infty(q^3;q^3)_\infty}.
\]
The polynomial identity
\begin{equation}\label{eq:norm-determinant-factorization}
 (XY-qZ^2)^2-(X^2+qYZ)(Y^2+XZ)
       =-Z\mathcal N(X,Y,Z)
\end{equation}
therefore proves the lemma.
Here we cancel \(\mathsf C\), whose constant coefficient is \(1\).
\end{proof}

\begin{lemma}\label{lem:KR-lower-bounds}
The following holds.
\begin{equation}\label{eq:KR-lower-bounds}
 \mathsf A-\KR_1,\quad \mathsf B-\KR_2,\quad
 \mathsf C-\KR_3\ \in q\mathbb Z_{\geq0}[[q]].
\end{equation}
\end{lemma}

\begin{proof}
Let \(\mathcal D_i\), \(i=1,2,3\), be the partition sets in the introduction.
Kur\c{s}ung\"oz proves that the generating functions
\(\sum_{\lambda\in\mathcal D_i}q^{|\lambda|}\) are
\(\mathsf A,\mathsf B,\mathsf C\), respectively
\cite{Kursungoz2019};
in his two-variable formulas, the length variable is set to \(1\).

By Tsuchioka~\cite[Corollary 1.4]{Tsuchioka2022}, for every \(N\geq0\),
the number of partitions of \(N\) in \(\mathcal D_i\) is at least
the number of partitions of \(N\) into parts congruent to
\(2^{i-1},3,6,9-2^{i-1}\pmod9\).
The latter generating functions are \(\KR_1,\KR_2,\KR_3\).
All six series have constant coefficient \(1\), proving
\eqref{eq:KR-lower-bounds}.
\end{proof}

\begin{lemma}
\label{lem:positive-norm-rigidity}
Let \(X,Y,Z,x,y,z\in1+q\mathbb Z[[q]]\).  If
\(X-x,Y-y,Z-z\) have nonnegative coefficients and
\(\mathcal N(X,Y,Z)=\mathcal N(x,y,z)\), then
\((X,Y,Z)=(x,y,z)\).
\end{lemma}

\begin{proof}
Write \(d_1=X-x\), \(d_2=Y-y\), and \(d_3=Z-z\).
Telescoping the cubic and mixed terms gives the exact identity
\begin{equation}\label{eq:positive-norm-factorization}
\begin{aligned}
 \mathcal N(X,Y,Z)-\mathcal N(x,y,z)
 ={}&d_1\bigl(X^2+Xx+x^2+3qYZ\bigr)\\
 &+qd_2\bigl(Y^2+Yy+y^2+3xZ\bigr)\\
 &+qd_3\bigl(3xy-q(Z^2+Zz+z^2)\bigr).
\end{aligned}
\end{equation}
The three expressions in parentheses have constant coefficients
\(3,6,3\), respectively.  Suppose that at least one \(d_i\) is
nonzero, and define
\[
 \nu=\min\{\ord_qd_1,\ 1+\ord_qd_2,\ 1+\ord_qd_3\},
 \qquad \ord_q0=+\infty.
\]
Each summand of order \(\nu\) on the right of
\eqref{eq:positive-norm-factorization} has a positive coefficient
at that order.  Every other summand has greater order.  Thus the
coefficient of \(q^\nu\) is positive, contradicting
\(\mathcal N(X,Y,Z)=\mathcal N(x,y,z)\).  All three differences vanish.
\end{proof}

\begin{proof}[Proof of Theorem~\ref{thm:KR-evaluations}]
By Lemma~\ref{lem:equal-norms} and Proposition~\ref{prop:mc-product-cubic},
\(\mathcal N(\mathsf A,\mathsf B,\mathsf C)=\mathcal N(\KR_1,\KR_2,\KR_3)\).
Apply Lemma~\ref{lem:positive-norm-rigidity} to this equality and the lower bounds in
Lemma~\ref{lem:KR-lower-bounds}.
\end{proof}

\section*{Use of AI}
\label{sec:ai-discovery}

I would first like to describe how I came to ask AI to solve this conjecture.
I learned about the conjecture while I was a student in the Department of Mathematical and Computing Science
at Tokyo Institute of Technology (now Institute of Science Tokyo), and I remember hearing about it from Shunsuke Tsuchioka.
I do not remember exactly when I first heard of it, but hearing about it in Tsuchioka's lectures
(on Rogers--Ramanujan identities and related topics), which I attended via Zoom in 2020, left a lasting impression on me.

When I saw the conjecture, I noticed that the $q$-series on the sum side could be realized as
partition $q$-series of Kato and Terashima~\cite{KatoTerashima2015} in the context of cluster algebras,
which I was studying at the time (Terashima was then my supervisor).
This realization could be used to derive related dilogarithm identities, but it did not provide a way to prove
the original $q$-series identities. Ever since then, the conjecture had remained on my mind.
Its connection with cluster algebras was also one of the main motivations for my later paper~\cite{Mizuno2025}.
In that paper, I stated the modular transformation formulas for the products without proof.
I had in fact obtained a proof of the transformation formulas by expressing the products as double theta series,
but had not written it up because it lay outside the main subject of the paper and I thought it might be useful
for the Kanade--Russell conjecture itself.
Wang and Zhang~\cite{WangZhang2025} later proved these transformation formulas with a shorter proof than mine,
so in the end, I think it was good that I had left my proof out.

Because this history had kept the problem on my mind, I tried asking each new AI model whether it could solve the conjecture.
At first, I gave the models my own previous attempts (mostly recorded in SageMath files) and asked whether they could solve it.
In particular, I proposed relating the theta series on the product side to the sum side.
It did not take long for the files generated during their attempts to outnumber my own.
The first model to make progress that led directly to the present proof was GPT-5.6 Sol, which discovered on July 22, 2026 that the series
on the sum side could be expressed as a Casoratian, as in \eqref{eq:app-Casoratian}.
While exploring this direction, they arrived at the auxiliary series
\[
 V(x,z)=(q^2/z;q^3)_\infty
 \sum_{n\in\mathbb Z}\frac{q^{3n^2-2n}z^{-2n}}{(q^2/z;q^3)_n}
 \mathcal R_q(-q^{1-3n}xz),
\]
which incorporates the Ramanujan function into a bilateral Rogers--Ramanujan series
\cite[Theorem 2.1]{Schlosser2023}.
For negative integers \(n\), we use
\((x;q)_n=(x;q)_\infty/(xq^n;q)_\infty\).
We have
\[
 V(x,z)=a(z)\mathcal A(x)+b(z)\mathcal B(x),\qquad
 a(z)=\theta_{q^6}(qz^{-2}),\quad b(z)=-z^{-1}\theta_{q^6}(q^4z^{-2}).
\]
Hence, for two values \(z_1,z_2\),
\[
 \begin{pmatrix}
 V(x,z_1)&V(x,z_2)\\
 V(qx,z_1)&V(qx,z_2)
 \end{pmatrix}
 =\begin{pmatrix}
 \mathcal A(x)&\mathcal B(x)\\
 \mathcal A(qx)&\mathcal B(qx)
 \end{pmatrix}
 \begin{pmatrix}
 a(z_1)&a(z_2)\\
 b(z_1)&b(z_2)
 \end{pmatrix}.
\]
On the right, the dependence on \(x\) and that on \(z_1,z_2\) occur in separate factors.
The determinant of the first factor is the Casoratian of \(\mathcal A,\mathcal B\),
and studying it led to the representation \eqref{eq:app-Casoratian}.
Dividing it by \(x\) gives a function satisfying \eqref{eq:app-exterior-equation},
with constant term \(1\) by \eqref{eq:app-parity-initial-determinant}.
These properties identify it with \(\mathsf F(q^2x,q^3x^3)\).
This did not yet lead to a complete proof.
When GPT-6 Astra became available in September 2026, I asked, as usual, whether they could solve the conjecture.
This time, the attempt led to a complete proof.
The key discovery was the equality \(\mathcal N(\mathsf A,\mathsf B,\mathsf C)=\mathcal N(\KR_1,\KR_2,\KR_3)\).
A clue came from Wang and Wang's conjectural expressions for special values of \(\mathsf G\) in terms of infinite products~\cite[Conjecture 3.7]{WangWang2026}:
after simplification, their right-hand sides are the three quadratic expressions
in \eqref{eq:app-quadratic-bridge}, with the products \(\KR_1,\KR_2,\KR_3\) in place of the sums.
This suggested the identities \eqref{eq:app-quadratic-bridge}, which GPT-6 Astra proved using Proposition~\ref{prop:app-addition}.
They then took determinants in the matrix form \eqref{eq:app-source-matrix} of the connection formula
and used the Casoratian identity to obtain a relation among the three values of \(\mathsf G\).
Substituting \eqref{eq:app-quadratic-bridge} into this relation, its left-hand side factors as
\(-\mathsf C\mathcal N(\mathsf A,\mathsf B,\mathsf C)\) by \eqref{eq:norm-determinant-factorization},
revealing the cubic \(\mathcal N\).
This cubic essentially gives a theta function, which connects it to the product side.
Indeed, as explained below, it arises naturally from the modular transformation formulas for the products.
Put $q=e^{2\pi i\tau}$ with $\operatorname{Im}\tau>0$,
$\zeta=e^{2\pi i/9}$, and $\omega=\zeta^3$. The column vector
\[
 f(\tau)=(q^{-1/18}\KR_1,q^{5/18}\KR_2,q^{11/18}\KR_3)^{\mathsf T}
\]
satisfies $f(-1/(3\tau))=Sf(\tau)$~\cite{Mizuno2025,WangZhang2025}, where
\[
 S=\begin{pmatrix}
 \alpha_1&\alpha_2&\alpha_4\\
 \alpha_2&-\alpha_4&-\alpha_1\\
 \alpha_4&-\alpha_1&\alpha_2
 \end{pmatrix},\qquad
 \alpha_j=\frac{1}{2\sqrt3\sin(j\pi/9)}.
\]
For the linear forms $L_j:\mathbb C^3\to\mathbb C$,
$L_j(X,Y,Z)=X+\omega^jY-\omega^{2j}Z$ $(j=0,1,2)$, we have
\[
 \begin{gathered}
 L_0L_1L_2=X^3+Y^3-Z^3+3XYZ,\\
 L_0\circ S=L_0,\qquad
 L_1\circ S=\zeta L_2,\qquad
 L_2\circ S=\zeta^{-1}L_1.
 \end{gathered}
\]
Thus their product is invariant under $S$. Evaluating it at $f(\tau)$ gives
\[
 L_0(f(\tau))L_1(f(\tau))L_2(f(\tau))
 =q^{-1/6}\mathcal N(\KR_1,\KR_2,\KR_3).
\]
This recovers the cubic $\mathcal N$.
These transformation laws, combined with a comparison of the expansions at the cusps,
also give a modular proof of Proposition~\ref{prop:mc-product-cubic}.

This showed that the sum and product sides give the same value, establishing a clear relation between them.
GPT-6 Astra then combined this equality with Tsuchioka's coefficientwise inequalities
(Lemma~\ref{lem:KR-lower-bounds}) as the final piece of the proof, thereby proving that the two sides coincide.

GPT-6 Astra also formalized the three Kanade--Russell identities in Lean 4 based on this proof~\cite{palomar-2026-09-10-000002-v1}.

\printbibliography
\bigskip

\noindent
\textsc{Yuma Mizuno, School of Mathematical Sciences, University College Cork, Western Road,
Cork, Ireland.}\par
\noindent
\textit{Email address}: \texttt{YMizuno@ucc.ie}

\end{document}